\documentclass[12pt]{amsart}
\usepackage{amssymb}
\usepackage{amscd}
\usepackage{amsxtra}
\usepackage{amsmath}
\usepackage[mathscr]{eucal}
\usepackage{tikz-cd, enumitem}
\usepackage{fullpage}

\theoremstyle{plain}
\newtheorem{thm}{Theorem}[section]
\newtheorem{lem}[thm]{Lemma}
\newtheorem{prop}[thm]{Proposition}

\theoremstyle{definition}
\newtheorem{defn}[thm]{Definition}

\newtheorem{ques}[thm]{Question}
\newtheorem{exmp}[thm]{Example}

\theoremstyle{remark}
\newtheorem*{rem}{Remark}

\newcommand{\N}{\mathbb{N}}

\newcommand{\Z}{\mathbb{Z}}

\newcommand{\R}{\mathbb{R}}

\newcommand{\Core}{\mathsf{Core}}

\newcommand{\SL}{\mathsf{SL}}

\newcommand{\stab}{\mathsf{Stab}}
\newcommand{\dep}{\mathsf{depth_{RF}}}

\title{Generalized Residual Finiteness of Groups}
\author{Nic Brody \and Kasia Jankiewicz }
\address{Department of Mathematics, University of California, Santa Cruz, CA 95064}
\email{nic@ucsc.edu}
\email{kasia@ucsc.edu}

\makeatletter
\@namedef{subjclassname@2020}{\textup{2020} Mathematics Subject Classification}
\makeatother


\begin{document}

\begin{abstract}
	A countable group is residually finite if every nontrivial element can act nontrivially on a finite set. When a group fails to be residually finite, we might want to measure how drastically it fails - it could be that only finitely many conjugacy classes of elements fail to act nontrivially on a finite set, or it could be that the group has no nontrivial actions on finite sets whatsoever. We define a hierarchy of properties, and construct groups which become arbitrarily complicated in this sense.
\end{abstract}

\maketitle

\section{Introduction}
Many infinite discrete groups are known to be residually finite. For example, free groups, and more generally, by the theorem of Mal'cev \cite{Malcev40}, all finitely generated linear groups are residually finite. Other examples include all finitely generated nilpotent groups. A famous open problem of geometric group theory asks whether all Gromov-hyperbolic groups are residually finite \cite[Prob 1.15]{BestvinaProblemList}. Without the assumption of Gromov-hyperbolicity, there are also many examples of groups which are not residually finite. 

Of course if a group $G$ has no finite index subgroups at all, then $G$ is very far from being residually finite. This happens, for example, for the Higman group (see Example~\ref{exmp:Higman}) or any infinite simple group. 
However, there are also non residually finite groups that are very close to being residually finite, in the sense that the intersection of all finite-index subgroups is a finite nontrivial group (see Example~\ref{exmp:Deligne2}). We would like to distinguish between these possibilities.

We propose the notion of \emph{$\alpha$-residual finiteness} for arbitrary ordinal $\alpha$, which generalizes the notions of a finite group, and a residually finite group. For example, the Deligne group (defined in Example~\ref{exmp:Deligne1}), which is a non-residually finite extension of $\mathbb Z$ by a residually finite group, is $(\omega\cdot 2)$-residually finite, where $\omega$ is the order type of the natural numbers, and so $\omega\cdot 2$ is the order type corresponding to two copies of natural numbers listed one after the other. On the other hand, the Higman group, whose intersection of all finite index subgroups is an infinite simple group, is not $\alpha$-residually finite for any ordinal $\alpha$. For a precise definition of $\alpha$-residual finiteness, see Section~\ref{sec:alpha-rf}.

Our main result is a construction of the following examples.
\begin{thm}\label{thm:intro construction}
For every $n\in \mathbb Z$, where $n\geq 1$, there exists a finitely generated group $G_n$ which is $\omega\cdot n$-residually finite, but not $\omega\cdot(n-1)$-residually finite.
\end{thm}

We also give a characterization of the $\alpha$-residual finiteness in terms of actions on rooted $\alpha$-trees, which can be thought of as trees of depth $\alpha$. Informally, those are collections of vertex sets and edge sets indexed by ordinals $i\leq\alpha$, with edges joining vertices in sets whose indices differ by 1.
For limit ordinals, the vertex sets are defined as the limit sets of the preceding sets of vertices.
\begin{thm}\label{thm:intro characterization}
A group $G$ is $\alpha$-residually finite if and only if $G$ admits a simple action on a rooted $\alpha$-tree. 
\end{thm}

This note is organized as follows. In Section~\ref{sec:ordinals} we recall some background on ordinals and cardinals. In Section~\ref{sec:alpha-rf} we give motivation and definition of $\alpha$-residual finiteness, and prove some properties of this notion. In the following Section~\ref{sec:trees} we include a discussion on $\alpha$-trees and prove Theorem~\ref{thm:intro characterization}.
Finally, in Section~\ref{sec:construction} we prove Theorem~\ref{thm:intro construction} 

\subsection*{Acknowledgements} 
We thank Martin Bridson and Marco Linton for their helpful comments.
The second author was supported by the NSF grants DMS-2203307 and DMS-2238198. 

\section{Background on ordinals and cardinals}\label{sec:ordinals}
We include basics on ordinals and cardinals. For more background see e.g.\ \cite{Halmos74}.
\subsection{Ordinals} An \emph{ordered set} $(X,\leq)$ consists of a set $X$, and a binary relation $\leq$, which is reflexive, anti-symmetric, and transitive. An ordered set $(X,\leq)$ is \emph{well-ordered}, if for any $a,b\in X$ either $a\leq b$ or $b\leq a$, and every non-empty subset of $X$ has a least element with respect to $\leq$.

Let $(X, \leq)$ and $(Y,\leq)$ be two ordered sets. A function $f:X\to Y$ is \emph{monotonic} if for every $a,b \in X$ such that $a\leq b$, we have $f(a)\leq f(b)$. An \emph{order isomorphism} is a monotonic bijection whose inverse is also monotonic. We say $(X, \leq)$ and $(Y,\leq)$ have the same \emph{order type} if there exists an order isomorphism between $X$ and $Y$. We note that having the same order type is an equivalence relation. 

An \emph{ordinal} is the order type of a well-ordered set.
The ordinal $\omega$ is the order type of the natural numbers with the standard order $\leq$. Every natural number $n$ is the order type of the set $\{1, 2,\dots, n\}$ with the standard order.

\subsection{Ordinal arithmetic}
The arithmetic operation of addition and multiplication can be defined for ordinals. Let $\alpha, \beta$ be two ordinals. We define $\alpha+\beta$ to be the ordinal whose underlying set is the disjoint union of those of $\alpha$ and $\beta$, and the order is extended so that each element of $\alpha$ is less than each element of $\beta$.
We let $\alpha\cdot \beta$ be the ordinal whose underlying set is the product of those for $\alpha$ and $\beta$, and the order is \emph{reverse} lexicographic, so that $(a_1,b_1)<(a_2,b_2)$ if and only if $b_1<b_2$ or $b_1=b_2$ and $a_1<a_2$.

For example, the ordinal $\omega + \omega$ corresponds to the order type of two copies of the natural numbers, where each number in the first copy is smaller than each number in the second copy. This is the same ordinal as $\omega \cdot 2$.

We note that neither addition nor multiplication are commutative. For example, $1+\omega = \omega$, but $\omega+1$ is the order type of the ordered set $(\mathbb N \cup \{\infty\}, \leq)$ where the order on $\mathbb N$ is standard, and $n\leq \infty$ for every $n\in \mathbb N$. In particular, the order type $\omega+1$ contains a largest element, while the order type $\omega$ does not.
Similarly, $\omega\cdot 2 \neq 2\cdot \omega$.

An ordinal $\alpha$ is a \emph{successor} of $\beta$ if $\alpha$ is the smallest ordinal greater than $\beta$, i.e.\ $\alpha = \beta +1$. 
A \emph{limit ordinal} is a non-zero ordinal that is not a successor ordinal. 
Equivalently, $\alpha$ is a limit ordinal if there exists $\beta$ such that $\beta<\alpha$, and for every such $\beta$, there exists an ordinal $\gamma$ such that $\beta<\gamma<\alpha$.
Note that, in particular, every successor ordinal is of the form $\alpha + n$ for some limit ordinal $\alpha$ and $n>0$.

\subsection{Cardinals}
A \emph{cardinal} is a set considered up to bijection. There is a natural association of a cardinal to each ordinal, by taking any set of given order type and consider it up to bijection. Finite ordinals are in one to one correspondence with finite cardinals, and both can be identified with natural numbers. Among infinite ordinals, there are many ordinals that correspond to the same cardinal; ordinals $\omega$, $\omega+1$ and $\omega \cdot 2$ all correspond to the cardinal $\aleph_0$.

\section{$(\alpha,\kappa)$-residual finiteness}\label{sec:alpha-rf}
We recall that a countable group $G$ is \emph{residually finite} if for every nontrivial $g\in G$, there exists a finite index subgroup $H\subseteq G$ such that $g\notin H$. Equivalently, $G$ is residually finite if for every nontrivial $g\in G$ there exists a finite quotient $\phi:G\to Q$ such that $\phi(g)\neq 1$.
The \emph{residual finiteness core} $\Core(G)$ of a countable group $G$, is the intersection of all finite index subgroups of $G$. A group $G$ is residually finite if and only if $\Core(G) = \{1\}$.

\subsection{Motivation}
We start with reviewing some examples of non-residually finite groups.
\begin{exmp}[\cite{Higman}]\label{exmp:Higman}
The Higman group is given by the presentation
\[H = \langle a,b,c,d\mid a^{-1}ba=b^{2},\quad b^{-1}cb=c^{2},\quad c^{-1}dc=d^{2},\quad d^{-1}ad=a^{2}\rangle.\]
It is an infinite group that admits no finite quotients. 
In particular, $\Core(H) = H$. 
\end{exmp}

\begin{exmp}[\cite{Deligne78}, see also \cite{morris2009lattice}]\label{exmp:Deligne1}
Note that the fundamental group of $Sp_{2n}(\R)$ is $\Z$. There exists a finite index subgroup $G\subseteq Sp_{2n}(\mathbb Z)$ such that the preimage $\widetilde G$ of $G$ in the universal cover $\widetilde{Sp_{2n}(\mathbb R)}$ of $Sp_{2n}(\mathbb R)$ is a central extension
\[
1\to \mathbb Z\to \widetilde G\to G\to 1.
\]
Moreover every finite index subgroup of $\widetilde G$ contains the kernel $\mathbb Z$, and in fact $\Core(G)$ is equal to the index two subgroup $2\mathbb Z$ of the kernel $\mathbb Z$. In particular, $\widetilde G$ is not residually finite, but \{residually finite\}-by-\{residually finite\}.
\end{exmp}

\begin{exmp}\label{exmp:Deligne2}
If instead of lifting $G$ to the universal cover, we lift to a finite cover of degree $k\geq 3$, we obtain a central extension of the form

\[1 \to \mathbb Z/k\mathbb Z\to \widehat G\to G \to 1.\]
The group $\widehat G$ is not residually finite, but finite-by-\{residually finite\}.
 \end{exmp}

Our goal is to distinguish the above groups using a finer notion than residual finiteness, which we define in the next section.

\subsection{Definition}
\begin{defn}\label{defn:alpha rf}
	Let $\alpha$ be an ordinal and $\kappa$ a cardinal. A group $G$ is called \emph{$(\alpha,\kappa)$-residually finite} if there exists an $\alpha$-indexed chain $\{\mathscr{C}_i(G)\}_{i\leq\alpha}$ of subgroups of $G$ so that 
    \begin{enumerate}[label=(\roman*)]
        \item $\mathscr{C}_0(G) = G$, and $\mathscr{C}_\alpha(G)=\{1\}$,
        \item $[\mathscr{C}_i(G):\mathscr{C}_{i+1}(G)]<\kappa$ for all $i<\alpha$,
        \item $\mathscr{C}_\lambda(G) =\bigcap_{i<\lambda}\mathscr{C}_i(G)$ for limit ordinals $\lambda \leq \alpha$.
    \end{enumerate}
\end{defn}

When $(\alpha,\kappa)=(\omega,\aleph_0)$, we recover the standard notion of residual finiteness.
Groups which are $(\alpha,\aleph_0)$-residual will be called \emph{$\alpha$-residually finite}. 
We note that if $G$ is $\alpha$-residually finite, then $G$ is $\beta$-residually finite for every $\beta>\alpha$.

Similarly, if $G$ is $(\alpha,\kappa_1)$-residually finite, then $G$ is $(\alpha,\kappa_2)$-residually finite for every $\kappa_2>\kappa_1$. Moreover, every group $G$ is $(1,\kappa)$-residually finite for every $\kappa > |G|$. However, if $p$ is prime and $k<p$, $\Z/p\Z$ is not $(\alpha,k)$-residually finite for any $\alpha$.

\begin{defn}
    We say the \emph{residual finiteness depth} $\dep(G)=\alpha$ if $G$ is $\alpha$-residually finite, but not $\beta$-residually finite for any $\beta<\alpha$
\end{defn}

While our definition does not appear in the literature, Marco Linton informed us of the following result of Baumslag \cite{Baumslag71}. In our terminology, this result can be phrased as $\omega\cdot n$-residual finiteness of all positive one-relator groups (i.e.\ where the relator is a positive word). Baumslag also conjectured that all one-relator groups are $\omega\cdot n$-residually finite for some $n$ \cite{Baumslag74}.

\subsection{Properties}
The notion of $\alpha$-residual finiteness is not very interesting for finite ordinals. Indeed, for every $n\in \mathbb N$ such that $n\geq 1$ a group $G$ is $n$-residually finite if and only if $G$ is finite. More generally, we have the following.

\begin{prop}\label{prop:limit or successor}
    If $\dep(G)=\alpha$, then $\alpha$ is 0, 1, a limit ordinal, or the successor of a limit ordinal.
\end{prop}

\begin{proof} Any successor ordinal can be expressed as $\alpha + n$ for some $n\geq 1$, and some limit ordinal $\alpha$. Suppose that $G$ is $(\alpha+n)$-residually finite for some $n\geq 2$ and sone limit ordinal $\alpha$, and let $\{\mathscr{C}_{i}(G)\}_{i\leq \alpha+n}$ be its $\alpha+n$ index chain provided by the definition. 
That means that $\mathscr{C}_{\alpha+n}(G) =\{1\}$  and $[\mathscr{C}_{\alpha+n-1}(G): \mathscr{C}_{\alpha+n}(G)]< \infty$,\dots, $[\mathscr{C}_{\alpha}(G): \mathscr{C}_{\alpha+1}(G)]< \infty$, hence $[\mathscr{C}_{\alpha}(G): \mathscr{C}_{\alpha+n}(G)]< \infty$. In particular the chain $\{\mathscr{C}'_{i}(G)\}_{i\leq \alpha+1}$ where $\mathscr{C}'_{i}(G) = \mathscr{C}_{i}(G)$ for $i\leq \alpha$, and $\mathscr{C}'_{\alpha+1}(G) = \mathscr{C}_{\alpha+n}(G) = \{1\}$ is an $(\alpha+1, \aleph_0)$-residual chain for $G$ as in Definition~\ref{defn:alpha rf}. Thus $G$ is $(\alpha+1)$-residually finite.
\end{proof}

\begin{exmp}
    We have $\dep(G)=0$ if and only if $G$ is the trivial group, and $\dep(G) = 1$ if and only if $G$ is a nontrivial finite group. A group $G$ with $\dep(G) = \omega + 1$ is finite-by-residually finite.
\end{exmp}
More generally, we have the following.
\begin{prop}\label{prop:extensions}
	Suppose that $$1 \to N\to G\xrightarrow{\pi} Q\to 1$$ is a short exact sequence of groups where $Q$ is $(\alpha_1,\kappa_1)$-residual and $N$ is $(\alpha_2,\kappa_2)$-residually finite. Then $G$ is $(\alpha_1+\alpha_2,\max\{\kappa_1,\kappa_2\})$-residually finite.
	
	\begin{proof}
		Let $\{\mathscr{C}_i(Q)\}_{i\leq \alpha_1}$, $\{\mathscr{C}_i(N)\}_{i\leq \alpha_2}$ be $(\alpha_1,\kappa_1)$- and $(\alpha_2,\kappa_2)$-residual chains for $Q$ and $N$ respectively. For $i\leq \alpha_1+\alpha_2$ set $\mathscr{C}_i(G)=\begin{cases} \pi^{-1}(\mathscr{C}_i(Q)) &\text{ for }i\leq \alpha_1 \\ \mathscr{C}_{i-\alpha_1}(N) &\text{ for }i> \alpha_1\end{cases}$
		
		Then $\mathscr{C}_{\alpha_1}(G) = \bigcap_{i<\alpha_1}\pi^{-1}(\mathscr{C}_i(Q))=\pi^{-1}(\bigcap_{i<\alpha_1}\mathscr{C}_i(Q))=\pi^{-1}(1_Q)=N $. 
        Thus $\{\mathscr{C}_{i}(G)\}_{i\leq \alpha_1+\alpha_2}$ is an $(\alpha_1+\alpha_2, \max\{\kappa_1, \kappa_2\})-$residual chain for $G$.
	\end{proof}
\end{prop}

\begin{prop}\label{prop:core}
    Let $G$ be a group such that $\Core(G)$ has $\dep(\Core(G)))=\alpha>0$. Then $\dep(G) = \omega+\alpha$.
\end{prop}
\begin{proof}
   Note that if $\Core(G)$ is $\alpha$-residually finite for some ordinal $\alpha>0$, then necessarily $[G:\Core(G)]=\infty$. Indeed, if $[G:\Core(G)]<\infty$, then any finite index subgroup of $\Core(G)$, which exists by the assumption that $\Core(G)$ is $\alpha$-residually finite, would also have finite index in $G$, contradiction the definition of $\Core(G)$. Thus, $G/\Core(G)$ is an infinite residually finite group. By Proposition~\ref{prop:extensions} $G$ is $\omega+\alpha$-residually finite. It remains to prove that $\dep(G)$ is not less that $\omega+\alpha$.

   Since $G$ surjects onto an infinite group, clearly $G$ is infinite, so $\dep(G)\geq \omega$. Suppose that $G$ is $\omega+\beta$-residually finite for some ordinal $\beta$, and let $\{\mathscr{C}_i(G)\}_{i\leq \alpha}$ be the $\omega+\beta$-indexed chain witnessing the $\omega+\beta$-residual finiteness of $G$. Since $\mathscr{C}_{\omega}(G)$ is an intersection of finite index subgroups of $G$ we have $\mathscr{C}_{\omega}(G)\supseteq \Core(G)$. By construction, this $\omega+\beta$-indexed chain also provides a $\beta$-indexed chain witnessing $\beta$-residual finiteness of $\mathscr{C}_{\omega}(G)$. In particular, this proves that $\Core(G)$ is $\beta$-residually finite. 

   It follows that $\dep(G) = \omega+\alpha$ as claimed.
\end{proof}

\begin{rem}
    We emphasize that $\omega+\alpha$ might be equal to $\alpha$. Indeed, this is the case precisely when $\alpha\geq \omega^\omega$.
\end{rem}

In the next proposition, $G^X$ denotes the group of function from the set $X$ to the group $G$, with the group operation defined coordinate-wise. By $G^{(X)}$ we denote the subgroup of the functions with finite support.

\begin{prop}\label{prop:G^X}
	If $G$ is $\alpha$-residually finite, 
  and $X$ is a countable set, $G^X$ is $\alpha$-residually finite. In particular, if $G$ is infinite, $\dep(G)=\dep(G^X)=\dep(G^{(X)})$.
 
	\begin{proof}
		Let $\{\mathscr{C}_i(G)\}_{i\leq \alpha}$ be a residual chain for $G$, and $\{x_0,x_1,x_2,x_3,\dots\}$ an enumeration of $X$. Then consider the chain $H_n=\{f\colon X\to G\mid f(x_i)\in \mathscr{C}_{n-i}(G) \text{ for } i<n\}$. This has index $[G:\mathscr{C}_1(G)][G:\mathscr{C}_2(G)]\dots [G:\mathscr{C}_n(G)]<\infty$, and the intersection of all $H_n$ is trivial.  

      The second statement follows because $G\leq G^{(X)}\leq G^X$.
	\end{proof}
\end{prop}

\section{Actions on rooted trees}\label{sec:trees}

\subsection{Faithful actions on rooted finite valence trees}
If $G$ is a group, $g\in G$ and $H\leq G$, the element $g$ permutes the left cosets of $H$. For a fixed group $G$, we may wish to study all such actions at once, for all $g$ and all $H$. Note that if $H_2\leq H_1\leq G$, the permutations of $H_2$ are compatible with permutations of $H_1$, in the sense that if $g$ fixes $H_2$, it fixes $H_1$ as well.

\begin{prop}[Folklore]
	A countable group $G$ is residually finite if and only if it acts faithfully on a rooted finite valence tree.
	
	\begin{proof}
		Let $\mathscr{C}_i(G)$ for $i\in \N$ be a residual chain. Consider a rooted tree $T$ with the level $i$ vertex sets $V_i=G/\mathscr{C}_i(G)$, and edges joining the cosets $g\mathscr{C}_{i+1}(G)$ and $g\mathscr{C}_i(G)$. This admits a left $G$-action via $h(g\mathscr{C}_i(G))=(hg)\mathscr{C}_i(G)$, which preserves the edge structure. Since $\{\mathscr{C}_i(G)\}_{i\in\N}$ is a residual chain, for every $1\not=g\in G$ there is an $i$ with $g\not\in \mathscr{C}_i(G)$. So $g$ acts nontrivially on $V_i$. This shows that $G$ acts faithfully on a rooted finite valence tree.
		
		Conversely, suppose $G$ acts faithfully on a rooted tree $T$, and let $G_i$ be the finite index subgroup which fixes the $i$th level of the tree. Since the action is faithful, no nontrivial element fixes every level of the tree, so $\cap_{i\in \N}G_i=1$.
	\end{proof}
\end{prop}

\subsection{Rooted $\alpha$-trees}
\begin{defn}
        Let $\alpha$ be an ordinal, and $\kappa$ a cardinal. 
	A \emph{rooted $(\alpha,\kappa)$-tree} $T$ is a family $\{V_i\}_{i\leq \alpha}$ of sets, and a family $\{E_i\}_{i<\alpha}$ of functions, where $V_0=\{*\}$, $E_i\colon V_{i+1}\to V_i$ with $|E_i^{-1}(v_i)|<\kappa$, and for a limit ordinal $\lambda\leq \alpha$, $V_\lambda =\varprojlim_{i<\lambda}V_i$. 
\end{defn} 
    For each $i\leq\alpha$, we refer to $V_i$ as the \emph{vertex set of $T$ at level $i$}. By $V_{<i}$ we denote the union $\bigcup_{j<i}V_j$, the \emph{vertex set of $T$ at level at most $i$}.
    Note that when $\alpha$ is a finite ordinal, then $T$ is just a rooted tree of depth $\alpha$. When $\alpha=\omega$, then $T$ is a standard infinite rooted tree, with an extra vertex corresponding to each end of $T$.

 The directed system determines restriction maps $E_i^j\colon V_j\to V_i$ for any $i\leq j\leq\alpha$. Indeed, if $j=i+n$ for some finite $n\geq 0$, then $E_i^j$ is the composition of finitely many maps $E_{j-1}\cdot E_{j-2} \cdots E_{i}$. Otherwise $j = \beta +n$ for some limit ordinal $\beta$ and some finite $n\geq 0$, and $i<\beta$. Then by the definition of $V_{\beta}$ there is a map $E_{i}^{\beta}:V_{\beta}\to V_i$, and we define $E_i^j = E_{\beta}^j\cdot E_i^{\beta}$.

\subsection{Actions on rooted $\alpha$-trees}
	An automorphism of a rooted $(\alpha,\kappa)$-tree is a family $g=\{g_i\}_{i\leq \alpha}$ of bijections of $V_i$ satisfying $E_ig_{i+1}=g_iE_i$. An automorphism is \emph{simple} if the action on $V_\alpha$ is fixed-point free; a group acts \emph{simply} on an $\alpha$-tree if every nontrivial element acts as a simple automorphism.

\begin{thm}
A group $G$ is $(\alpha,\kappa)$-residually finite if and only if $G$ has a simple action on a rooted $(\alpha,\kappa)$-tree.
\end{thm}

	\begin{proof}\item
		
  \begin{itemize}
      \item[$\implies$]
      
      Supposing $\{\mathscr{C}_i(G)\}_{i\leq\alpha}$ is an $(\alpha,\kappa)$-residual chain for $G$, we can build a rooted $(\alpha,\kappa)$-tree with a simple $G$-action as follows. Let $V_j=\{\{g_i\mathscr{C}_j(G)\}_{i\leq j}\mid g_{i+1}\mathscr{C}_i(G)=g_i\mathscr{C}_i(G)\text{ for }i<j\}$, and take $E_j\colon V_{j+1}\to V_j$ by $E_j(\{g_i\mathscr{C}_{j+1}(G)\}_{i\leq j+1})=\{g_i\mathscr{C}_j(G)\}_{i\leq j}$. Then we have $|E_i^{-1}(v_i)|=[\mathscr{C}_{i+1}(G):\mathscr{C}_i(G)]<\kappa$, and $V_\lambda=\varprojlim_{i<\lambda}V_i$, so we have constructed an $(\alpha,\kappa)$-tree. The left $G$-action $g\{g_i\mathscr{C}_k(G)\}_{i\leq j}=\{gg_i\mathscr{C}_j(G)\}_{i\leq j}$ respects the edge structure.
      
      \item[$\impliedby$]	
		Suppose $G$ acts simply on an $(\alpha,\kappa)$-tree $T=(V_i,E_i)_{i\leq \alpha}$. Let $v_\alpha\in V_\alpha$, and set $v_i=E_i^\alpha(v_\alpha)$ for any $i\leq \alpha$. Let $\mathscr{C}_i(G)=\stab(v_i)$. By simplicity of the action, $\mathscr{C}_{\alpha}(G)=1$, and since the tree is locally $\kappa$, $[\mathscr{C}_i(G):\mathscr{C}_{i+1}(G)]\leq |E_i^{-1}(v_i)|<\kappa$ for all $i<\alpha$. Finally, if $\lambda$ is a limit ordinal, we have $\cap_{i<\lambda}\mathscr{C}_i(G)=\cap_{i<\lambda} \stab(v_i)=\{g\in G\mid gv_i=v_i \text{ for all }i<\lambda\}=\stab(\varprojlim_{i<\lambda}v_i)=\stab(v_\lambda)=\mathscr{C}_{\lambda}(G)$, as desired. 
  \end{itemize}
  
	\end{proof}

\section{A construction of $\omega\cdot n$-residually finite groups}\label{sec:construction}

If $X$ is a set with a transitive $G$ action, and $K$ is a group, the \emph{(restricted) wreath product} $K\wr_XG$ is defined 
\[K\wr_XG=K^{(X)}\rtimes G,\] 
where $K^{(X)}$ is the group of functions $X\to K$ with finite support, and the action of $G$ on $K^{(X)}$ is by precomposing with the $G$-action on $X$. By $K\wr G$, we mean $K\wr_G G$, where $G$ acts on itself by left multiplication. 

Note that $K\wr G$ is the quotient of the free product $K*G$ by the family of relations $\{[gkg^{-1},k']=1\mid k,k'\in K, 1\not=g\in G\}$.

\begin{thm}\label{thm:construction}
Suppose $G$ is an infinite, finitely generated, residually finite group with finite abelianization. Let $G_1=G$, and $G_{i+1}=G_i\wr G$. The group $G_n$ is a finitely generated group with $\dep(G_n) = \omega \cdot n$.
\end{thm}

For example, we can take $G$ to be perfect (i.e.\ $G = [G,G]$), such as $G=\SL_n(\Z)$ for $n\geq 3$, or $G$ a cocompact hyperbolic triangle group with generators of relatively prime orders, or $G$ a free product of (at least 2 nontrivial) finite perfect groups. We can also take $G$ to be $\Z^n\rtimes A_n$, where $A_n$ permutes the coordinates of $\Z^n$, for $n\geq 5$.

For $g\in G$, we let $G_{i}^{(g)}$ denote the functions whose support is contained in $\{g\}$, so this is an isomorphic copy of $G_i$. Note that we have a short exact sequence $G_i^{(G)}\to G_{i+1}\xrightarrow{\pi} G$, and that the normal closure of $G_i^{(1)}$ in $G_{i+1}$ is $G_i^{(G)}$.

\begin{lem}\label{lem:core}
	$[G_i,G_i]^{(G)}\leq \Core(G_{i+1})\leq G_i^{(G)}$.
	
	\begin{proof}
		To see that $\Core(G_{i+1})\leq G_i^{(G)}$, note that there is a natural map $G_{i+1}=G_i^{(G)}\rtimes G\xrightarrow[]{\pi} G$ with kernel $G_i^{(G)}$. Since $G$ is residually finite, a residual chain for $G$ pulls back under the quotient to a residual chain in $G_{i+1}$ terminating in $G_i^{(G)}$.
		
		To prove that $\Core(G_{i+1})\geq [G_i, G_i]^{(G)}$, it suffices to show that if $H$ is a subgroup of $G_{i+1}$ with finite index, then $H$ contains $[G_i, G_i]^{(G)}$. Since the intersection of finitely many conjugates of $H$ yields a normal subgroup $N$ of finite index in $G_{i+1}$, it is enough to show that $N$ contains $[G_i,G_i]^{(1)}$, as this implies that $N$ contains its normal closure $[G_i, G_i]^{(G)}$.
  
        Note that $\pi(N)$ has finite index in $G$, and since $G$ is infinite, $\pi(N)$ is nontrivial. Let $g_0$ be a nontrivial element in $\pi(N)$, and $g_1$ a lift to $N$. Now for any $x,y\in G_i^{(1)}$, the elements $[x,y]$ and $[g_1xg_1^{-1},y]$ have the same image in $G_{i+1}/N$. But since $g_1xg_1^{-1}\in G_i^{(g_0)}$ and $y\in G_i^{(1)}$, we have $[g_1xg_1^{-1},y]=1$, and thus $[x,y]\in N$. This shows that $[G_i,G_i]^{(G)}\leq \Core(G_{i+1})$.
	\end{proof}	
\end{lem}

\begin{proof}[Proof of Theorem~\ref{thm:construction}]
We prove by induction that $G_n$ is a finitely generated group with $\dep(G_n) = \omega\cdot n$. The group $G_1 = G$ satisfies those conditions by assumption. Suppose that $G_{n-1}$ is finitely generated group with $\dep(G_{n-1}) = \omega\cdot (n-1)$. Then $G_n = G_{n_1}\wr G$ is generated by the generators of $G_{n-1}$, and the generators of $G$, in particular $G_n$ is finitely generated.
By Proposition~\ref{prop:core} $$\dep(G_n) =\omega +\dep(\Core(G_{n})).$$
We need to show that $\dep(\Core(G_{n})) = \omega\cdot (n-1)$ to conclude that $\dep(G_n) = \omega +\omega\cdot(n-1) = \omega\cdot n$.

By Lemma~\ref{lem:core} we have $$\dep([G_{n-1},G_{n-1}]^{(G)}) \leq \dep(\Core(G_{n})) \leq \dep(G_{n-1}^{(G)}).$$
Now, by Proposition~\ref{prop:G^X}, $\dep([G_{n-1},G_{n-1}]^{(G)}) = \dep([G_{n-1}, G_{n-1}])$, and similarly $\dep(G_{n-1}^{(G)}) = \dep(G_{n-1})$.
By the inductive assumption $\dep(G_{n-1}) = \omega\cdot(n-1)$, and since $[G_{n-1}, G_{n-1}]$ has finite index in $G_{n-1}$, we also have $\dep([G_{n-1}, G_{n-1}] = \omega\cdot(n-1)$. We conclude that $\dep(\Core(G_n)) = \omega\cdot(n-1)$.
\end{proof}

Note that the groups with the residual depth $\omega\cdot n$ constructed in Theorem~\ref{thm:construction} are finitely generated but not finitely presented. 
The Deligne group $\widetilde G$ from Example~\ref{exmp:Deligne1} $\dep(\widetilde G) = \omega\cdot 2$ is finitely presented.
We do not know finitely presented examples for $n>2$.

\begin{ques}
Does there exist a finitely presented group $G_n$ with $\dep(G_n) = \omega\cdot n$ for each $n\in \mathbb N$?
\end{ques}

\begin{ques}
Does there exist a finitely generated (finitely presented?) group $\widehat G_n$ with $\dep(\widehat G_n) = \omega\cdot n +1$ for each $n\in \mathbb N$?
\end{ques}

There exist groups that are not $\alpha$-residually finite for any $\alpha$. An example of such a group is the Higman group (Example~\ref{exmp:Higman}) or any infinite simple group. However, we do not know whether there are finitely generated groups with the residual finiteness depth defined, but larger than $\omega\cdot n$ for all $n$. In particular, we do not know the answer to the following question.

\begin{ques}
Does there exist a finitely generated group $G$ with $\dep(G) = \omega^2$? What about $\omega^k$ for every $k\in N$? Can $G$ be chosen to be finitely presented?
\end{ques}

\section{Application to profinite rigidity}

A finitely generated, residually finite group $\Gamma$ is said to be \emph{ absolutely profinitely rigid} if for any finitely generated residually finite group $\Lambda$ with $\widehat\Gamma\cong \widehat\Lambda$, we have $\Gamma\cong\Lambda$. 

If there exists a group $\Gamma$ which is profinitely rigid and perfect, the group $\Gamma\wr\Gamma$ will have the same profinite completion as $\Gamma$, and so we would not be able to distinguish these two groups by their profinite completions. However, these groups may be distinguished in terms of $\alpha$-residual properties. 
Note that all of the groups which are known to be absolutely profinitely rigid (see \cite{bridson2020absolute} and \cite{cheethamwest2022absolute}) have finite nontrivial abelianization, so are not perfect.

\bibliographystyle{alpha}
\bibliography{main}

\end{document}